\documentclass[12pt,leqno]{article}

\usepackage{graphicx, fullpage}
\usepackage{amsmath,amssymb,amsthm,amscd, bm}
\usepackage{fancyhdr}
\usepackage[mathscr]{eucal}
\usepackage{amsfonts}
\begin{document}
\parskip=6pt

\theoremstyle{plain}
\newtheorem{prop}{Proposition}
\newtheorem{lem}[prop]{Lemma}
\newtheorem{thm}[prop]{Theorem}
\newtheorem{cor}[prop]{Corollary}
\newtheorem{defn}[prop]{Definition}
\theoremstyle{definition}
\newtheorem{example}[prop]{Example}
\theoremstyle{remark}
\newtheorem{remark}[prop]{Remark}
\numberwithin{prop}{section}
\numberwithin{equation}{section}
\def\cal{\mathcal}
\newcommand{\cF}{\cal F}
\newcommand{\cA}{\cal A}
\newcommand{\cC}{\cal C}
\newcommand{\cO}{{\cal O}}
\newcommand{\cE}{{\cal E}}
\newcommand{\cU}{{\cal U}}
\newcommand{\cM}{{\cal M}}
\newcommand{\cD}{{\cal D}}
\newcommand{\cK}{{\cal K}}
\newcommand{\cZ}{{\cal Z}}

\newcommand{\bC}{\mathbb C}
\newcommand{\bP}{\mathbb P}
\newcommand{\bN}{\mathbb N}
\newcommand{\bA}{\mathbb A}
\newcommand{\bR}{\mathbb R}
\newcommand{\oP}{\overline P}
\newcommand{\oQ}{\overline Q}
\newcommand{\oR}{\overline R}
\newcommand{\oS}{\overline S}
\newcommand{\oc}{\overline c}
\newcommand{\bp}{\mathbb p}
\newcommand{\oD}{\overline D}
\newcommand{\oE}{\overline E}
\newcommand{\oC}{\overline C}
\newcommand{\of}{\overline f}
\newcommand{\ou}{\overline u}
\newcommand{\ow}{\overline w}
\newcommand{\oy}{\overline y}
\newcommand{\oz}{\overline z}

\newcommand{\hg}{\hat G}
\newcommand{\hM}{\hat M}

\newcommand{\tpr}{\widetilde {\text{pr}}}
\newcommand{\tB}{\widetilde B}
\newcommand{\tx}{\widetilde x}
\newcommand{\ty}{\widetilde y}
\newcommand{\txi}{\widetilde \xi}
\newcommand{\teta}{\widetilde \eta}
\newcommand{\tna}{\widetilde \nabla}
\newcommand{\tth}{\widetilde \theta}

\newcommand{\var}{\varepsilon}
\newcommand{\vp}{\varphi}

\newcommand{\End}{\text{End }}
\newcommand{\loc}{\text{loc}}
\newcommand{\lam}{\lambda}
\newcommand{\Hom}{\text{Hom}}
\newcommand{\Ker}{\text{Ker}}
\newcommand{\dist}{\text{dist}}
\newcommand{\psl}{\rm{PSL}}
\newcommand{\rk}{\roman{rk }}
\newcommand{\id}{\text{id }}
\renewcommand\qed{ }
\begin{titlepage}
\title{On Riemannian submersions}
\author{L\'aszl\'o Lempert \thanks{Research partially  supported by NSF grant DMS1764167.  The paper was written at the R\'enyi Institute and the E\"otv\"os University, both of Budapest, while I was on sabbatical leave from Purdue University. I am grateful to all three institutions.
\newline
2010 Mathematics Subject Classification 53B20, 53C22, 32V05}\\ Department of  Mathematics\\
Purdue University\\West Lafayette, IN
47907-2067, USA}
\thispagestyle{empty}
\end{titlepage}
\date{}
\maketitle
\abstract
We prove that the image of a real analytic Riemannian manifold under a smooth Riemannian submersion is necessarily real analytic.
\endabstract

\section{Introduction}

A $C^1$ map $\vp$ between Riemannian manifolds $(X,h)$ and $(Y,k)$ is a Riemannian submersion if for every $x\in X$ the differential $\vp_*$ induces the metric of $T_{\vp(x)}Y$ from the metric of $T_xX$. In other words, $\vp_*$ maps the orthogonal complement of Ker$\,\vp_*$ isometrically on $TY$. Our main theorem is the following regularity result.

\begin{thm}  
Suppose $(X,h)$ is a real analytic and $(Y,k)$ is a $C^\infty$ Riemannian manifold, both finite dimensional. If 
$\vp:(X,h)\to(Y,k)$ is a surjective $C^\infty$ Riemannian submersion, then $Y$ admits a real analytic structure in which $k$ is real analytic.
\end{thm}  

Henceforward we will drop ``real'' from real analytic.---An amusing aspect of the theorem is that $\vp$ itself need not be 
analytic. Let $(X,h)$ be an analytic Riemannian manifold, and $S\subset X$ a smooth hypersurface that cuts $X$ in two. 
Then $\varphi(x)=\pm\text{dist}\,(x,S)$, $x\in X$, in a neighborhood of $S$  satisfies $|d\vp|\equiv 1$, and so defines a Riemannian submersion to $\bR$, which fails to be analytic unless $S$ is.

In Theorem 4.1 we will prove a more precise result, with $(Y,k)$ only $C^p$. The proof will use so called adapted complex structures. These structures first appeared in \cite{LSz1} as complex manifold structures on the tangent bundle of  Riemannian manifolds $(X,h)$ that are in a sense compatible with the metric. Simultaneously, Guillemin and Stenzel in \cite{GS} discovered an equivalent complex structure, but on the cotangent bundle. Guillemin--Stenzel and Sz\H{o}ke in \cite{Sz1, Sz3} prove that if $(X,h)$ is analytic, then on some neighborhood of the zero section in $T^*X$, resp. $TX$, an adapted complex structure exists; and this author in \cite{L1} proved the converse. Accordingly, we will derive Theorem 1.1 from a result that connects the existence of adapted complex structures for a Riemannian manifold and for its image under a Riemannian submersion. 

It turns out that it is better to approach this latter result from a perspective somewhat different from \cite{GS, LSz1}, and view adapted complex structures not on $TX$ or on $T^*X$, but on the related manifold $M_r$ of geodesics $x:[-r,r]\to X$, as suggested in \cite{LSz2}. Here $r$ is a positive number. For one thing, a Riemannian submersion $\vp:X\to Y$ induces a map $\vp_*:TX\to TY$, but $\vp_*$ cannot immediately be employed to transfer adapted complex structures from $TX$ to 
$TY$. Working with manifolds $M_r, N_r$ of geodesics in $X, Y$ though, it becomes obvious that $\vp$ does induce a 
submersion $\psi$ from a submanifold $H_r\subset M_r$ to $N_r$; namely from the submanifold of geodesics in $X$ 
perpendicular to the fibers of $\vp$. Now $H_r$ is not a complex submanifold of $M_r$. However, an adapted complex structure of $M_r$, wherever it exists, restricts to $H_r$ as a CR structure. This suggests that we enlarge the notion of adapted complex structures, and develop a quotient construction in this more general framework, that ultimately yields an adapted complex structure on $N^r$. We have already written in \cite{L2} on the generalization in question, a variant of which we will review in the next section. In Section 3 we discuss how an adapted complex structure associated with $(X,h)$ induces an adapted complex structure associated with its image $(Y,k)$ under a Riemannian submersion, and in Section 4 we prove Theorem 1.1 in a more precise form. 

\section{Adapted involutive structures}

Here we will review the notion of involutive structures adapted to the action of a certain monoid, a special instance of the notion studied in \cite{L2}. In one aspect we will be more general, though: while \cite{L2} worked with $C^\infty$ and more regular objects, in this paper we will work with $C^p$ H\"older classes, $p\in[0,\infty]$. At the same time, unlike in \cite{L2}, all our manifolds will be finite dimensional. If $E$ is a vector bundle, we abbreviate  $\bC\otimes_\bR E$ as $\bC E$. 

\begin{defn}  
Let $p\ge 1$ and $M$ a manifold of class $C^{p+1}$. An involutive structure of class $C^p$ on $M$ is a subbundle 
$P\subset\bC TM$ of class $C^p$ such that the Lie bracket of its $C^1$ sections is again a section. (Such a bundle itself
is also called involutive.)  A $C^1$ map $f:(M,P)\to (N,Q)$ of involutive manifolds is involutive if $f_*P\subset Q$.
\end{defn}   

For example, if $P\oplus\oP=\bC TM$, $P$ is an integrable almost complex structure. According to Hill  and Taylor 
\cite{HT}, when $p>1$ is not an integer, one can find $\bC$--valued local coordinates $z_1,\dots,z_m$ on $M$, of class 
$C^{p+1}$, such that $dz_j|\oP=0$. The case $p=\infty$ was first proved in \cite{NN}. The local coordinates turn $M$ into a 
complex manifold, whose bundle of $(1,0)$ vectors is $T^{10}M=P$. More generally, when $P\cap \oP=(0)$, $P$ defines a 
CR structure, that may or may not be embeddable in a complex manifold. 

While Definition 2.1 is made for manifolds $M$ without boundary, it makes sense even for manifolds with corners.  We will 
need involutive structures on one type of manifolds with corners: Fix $r\in (0,\infty)$ and consider affine maps 
$\bR\ni t\mapsto a+bt\in\bR$ that map $[-r,r]$ into itself. These maps form a monoid $G_r$, with the operation composition 
of maps. Let $e\in G_r$ denote the neutral element.
The coordinates $a,b$ identify $G_r$ with 
\[
\{a+ib: a,b\in\bR,\, |a|+|b|r\le r\},
\]
a diamond in $\bC$. We endow $G_r$ with the complex structure $S=T^{10}\bC |G_r$. 

\begin{defn}   
Let $M$ be a manifold of class $C^p$. A right action of $G_r$ on $M$, of class $C^p$, is a $C^p$ map
\[M\times G_r\ni (x,g)\mapsto xg\in M\]
such that $(xg)h=x(gh)$ and $xe=x$.
\end{defn}  

\begin{defn}   
Given a $C^1$ right action of $G_r$ on $M$, an involutive structure $P$ on $M$ is adapted (to $G_r$) if for every $x\in M$
\[G_r\ni g\mapsto xg\in M\]
is an involutive map between $({\rm{int}}\,G_r, S)$ and $(M,P)$.
\end{defn}   

Consider next a manifold $X$ of class $C^{p+1}$, $p\ge 2$, endowed with a Riemannian metric $h$ of class $C^p$. Given $r\in (0,\infty)$, let $M_r$ be the set of geodesics $x:[-r,r]\to X$ of $h$. For any $t\in[-r,r]$
\[M_r\ni x\mapsto \dot x(t)\in TX\]
is a bijection on an open subset of $TX$, 
and endows $M_r$ with a $C^p$ manifold structure, that is independent of the choice of $t$. If 
$(X,h)$ was analytic, so will be $M_r$.
In \cite{LSz1, LSz2} we focused on complete $(X,h)$ and the manifold $M_\infty$ of all geodesics $x:\bR\to X$. 
Of course, then all $M_r$ are the same, in the sense that the map $M_\infty\ni x\mapsto x|[-r, r]\in M_r$ is a diffeomorphism. 
In general, we embed $X$ into $M_r$ by associating  with $u\in X$ the constant geodesic $x\equiv u$, and view 
$X\subset M_r$ 
as a $C^p$ submanifold.

The monoid $G_r$ acts on $M_r$ on the right:
\begin{equation}
M_r\times G_r\ni (x,g)\mapsto xg=x\circ g\in M_r,
\end{equation}
and so on $G_r$ invariant submanifolds of $M_r$ we can talk about complex or involutive structures  that are adapted to 
$G_r$. The connection with adapted complex structures in the sense of \cite{LSz1, Sz3} is as follows. Embed $M_r$ into 
$TX$ by $x\mapsto \dot x(0)$. The action of $g(t)=a+bt$ on $M_r$ corresponds to following the geodesic flow on $TX$ for 
time $a$ and then fiberwise dilating by a factor of $b$. Note that the geodesic flow, as the flow of a $C^{p-1}$ vector field, is
$C^{p-1}$. The dilation action is also, hence the corresponding action (2.1) is $C^{p-1}$.
When $(X,h)$ is complete and $U$ is a neighborhood of the zero 
section in $TX=M_\infty$, \cite[Definition 4.1]{LSz1} defined a complex structure on $U$ adapted if the image $ug$ of any 
$u\in U$ under this map depends holomorphically on $a+bi$, and \cite{Sz3} extended this to geodesics of a not necessarily 
complete connection. This implies that under the embedding $M_r\hookrightarrow TX$ adapted complex structures in the 
sense of \cite{LSz1,Sz3} and of Definition 2.3 correspond.

The following is a special case of a result of Sz\H oke, \cite[Theorem 0.3]{Sz3}. It can also be derived from \cite{GS}.

\begin{thm}  
If $(X,h)$ is an analytic Riemannian manifold and $r\in (0,\infty)$, any constant geodesic $x\equiv u$ in $X$ admits a $G_r$ invariant neighborhood in $M_r$ and an analytic complex structure $P$ on it that is adapted to $G_r$.
\end{thm}   

Analyticity is not only sufficient but necessary for the existence of adapted complex structures: 

\begin{thm}   
Suppose $p>3$ is not an integer and $Y$ is a manifold of class $C^{p+1}$ endowed with a Riemannian metric $k$ of class $C^p$. Given $r\in(0,\infty)$, let $N_r$ stand for the $C^p$ manifold of geodesics $y:[-r,r]\to Y$, into which $Y$ is embedded as the submanifold of constant geodesics. If a $G_r$ invariant neighborhood of $Y\subset N_r$ admits an adapted complex structure $P$ of class $C^{p-1}$, then there is an analytic manifold structure on $Y$ in which $k$ is analytic, and whose underlying $C^p$ manifold is the same as that of the initial $C^{p+1}$ manifold. The same holds if $p\ge3$ is an integer, with the sole difference that instead of the underlying $C^p$ manifolds only the agreement of the underlying $C^{p-\var}$ manifold structures can be guaranteed, with arbitrary $\var>0$. 
\end{thm}    

\begin{proof}
When $(Y,k)$ is compact and $p=\infty$, we proved this in \cite{L1}, see Theorem 1.5 there. We indicate the modifications 
needed for the more general and precise statement. Let $U\subset N_r$ be a neighborhood of $Y$ on which the adapted complex structure $P$ exists. According to Hill and Taylor \cite{HT} locally on $U$ there are $C^p$ (or $C^{p-\var}$ if $p$ is an integer) functions $z_1,\dots,z_n$ that serve as holomorphic coordinates $(d\overline z_j | P=0)$. This endows $U$ with the structure of a complex manifold whose underlying $C^p$ (resp. $C^{p-\var}$) manifold agrees with that of the original 
structure of $U\subset TY$. Let $y^*(t)=y(-t)$. As in \cite[Theorem 5.7]{LSz1} it follows that 
$y\mapsto y^*$ is an antiholomorphic map in a neighborhood of $Y$. Hence $Y$, its fixed point set, 
is an analytic submanifold of the complex manifold $U$. What remains is to prove that  $k$ is analytic in the inherited 
analytic structure of $Y$. \cite{L1} does this by passing from the real quadratic form $k$ on $TY$ to the induced complex 
quadratic form $q$ on $T^{10}U | Y$ and producing a holomorphic extension of $q$ to $T^{10}U_1$, with $U_1\subset U$ 
a neighborhood of $Y$.

The extension is constructed via certain $n$ by $n$ matrices $f=(f_{\alpha\beta})$ of meromorphic functions associated with geodesics $y\in N_r$. In \cite{L1} the geodesics extend to complete geodesics $\bR\to Y$ and the $f_{\alpha\beta}$ are meromorphic on strips in $\bC$ about $\bR$, of width inversely proportional to the speed of the geodesic. In the case at hand all we can say is that if we stay in a sufficiently small neighborhood $V\subset U$ of a fixed constant geodesic $y_0\in Y\subset N_r$, any $y\in V$ will extend to a geodesic on an arbitrarily large interval $[-R,R]$. By choosing $V$ suitably we can arrange that the $f_{\alpha\beta}$ associated with $y\in V$ are meromorphic on the disc 
$D=\{s\in\bC:|s|<1\}$, with poles only on $\bR\cap D$. \cite[Lemma 6.7]{LSz1} generalizes to the current setting, and gives 
that $\pm \text{Im} \,f$ is positive definite on $D^\pm=\{s\in D: \pm \text{Im}\, s>0\}$. (An alternative proof is in \cite{Sz2}, cf. 
(4.5) there.) Further, $f(0)=0$.

To construct the holomorphic extension of $q$ to $T^{10}V_1$ with some neighborhood $V_1\subset V$ of $y_0$ we need a neighborhood of $0\in D$ free of zeros of $\det f'$, cf. Lemma 3.2 in \cite{L1}. Fix $\sigma\in(0,1)$ so that conjugate points along any $y\in V$ are at distance $>2\sigma$. As in \cite{L1}, this implies that the pole of $f^{-1}$ closest to the pole at 0, is at distance $>2\sigma$. Instead of finding a zero free region about 0 we will find one about $\sigma$. In view of the transformation formula (2.5) and Lemma 2.4 of \cite{L1}, this neighborhood can be transformed back to 0.

Now $f'(s)$ is invertible if $s\in(0,2\sigma)$ is not a pole of $f$ by the same argument as in \cite[Lemma 3.2]{L1}. For $s\in D^+$ it is easier to study the invertibility of $(f^{-1}){'}(s)$. One pushes $f^{-1}$ forward to a holomorphic matrix function $F$ on the upper half plane $\bC^+$ by a fixed conformal map $\vp:D^+\to \bC^+$. \cite[Proposition 3.1]{L1} guarantees a neighborhood of $\vp(\sigma)$ in $\overline{\bC^+}$ where $F'$ is invertible, hence a neighborhood of $\sigma$ in $\overline{D^+}$ where $(f^{-1}){'}$ is invertible. The neighborhood  is independent of $y\in V$. Replacing $D^+$ by $D^-$ we in fact obtain a neighborhood of $\sigma\in D$, independent of $y\in V$, where $0\ne \det(f^{-1})'=\det f'/(\det f)^2$. The upshot is that on a fixed neighborhood of $0\in D$ the matrices $f$ associated with geodesics $y\in V$ have det$\,f'\ne 0$ away from the poles of $f$. Once this granted, the rest of the proof in \cite{L1} goes through, and provides the holomorphic extension of $q$, whence the analyticity of $h$.
\end{proof}

\section{Riemannian submersions}

Now we come to the main point of the paper. Let $p\in[2,\infty]$, $X$ and $Y$ $C^{p+1}$ manifolds endowed with $C^p$
 metrics $h,k$ and $\vp:(X,h)\to(Y,k)$ a surjective Riemannian submersion of class $C^{p+1}$. Given $r\in(0,\infty)$, let 
 $M_r, N_r$ denote the $C^p$ manifolds of geodesics $x:[-r,r]\to X$, resp. $y:[-r,r]\to Y$, into which we embed $X, Y$ as 
 before. There is no natural map between $M_r$ and $N_r$ that $\vp$ induces. However, there is a submanifold 
 $H_r\subset M_r$ consisting of geodesics that are perpendicular to the fibers of $\vp$ (`horizontal geodesics'). According 
 to O'Neill, composition with $\vp$ induces a map $\psi:H_r\to N_r$, see [O, Chapter 7, Lemma 45]. In the identifications
\begin{equation}   
M_r\ni x\mapsto \dot x(0)\in TX, \qquad N_r\ni y\mapsto \dot y(0)\in TY,
\end{equation}   
$H_r$ corresponds to the orthogonal complement
\[H=(\Ker\, \vp_*)^\perp\subset TX,\]
intersected with a suitable neighborhood of the zero section, and $\psi$ corresponds to the restriction of $\Phi=\vp_* | H$ 
to this neighborhood. We see that $\psi$ is of class $C^p$. 

\begin{thm}  
If, under the above assumptions, a $G_r$ invariant neighborhood $M\subset M_r$ of $X$ admits an adapted complex structure $P$ of class $C^{p-1}$, then the neighborhood $N=\psi(M\cap H_r)\subset N_r$ of $Y$ also admits an adapted complex structure of class $C^{p-1}$.
\end{thm}   

In \cite{A} Aguilar proved a special case, when $(Y,k)$ is a quotient of $(X,h)$ by the action of a Lie group of isometries, and our proof below is partly inspired by his. We start with a general result on pushing forward involutive and adapted structures.

\begin{lem}  
Let $M,N$ be manifolds of class $C^2$, $P\subset \bC TM$ and $Q\subset \bC TN$ subbundles of class $C^1$, and $\psi: M\to N$ a $C^2$ surjection such that $P=\psi_*^{-1}Q$.
\begin{itemize}
\item[(a)] If $P$ is involutive, then so is $Q$.
\item[(b)] Suppose in addition that for some $r\in(0,\infty)$ $G_r$ acts on the right on both $M$ and $N$, and $\psi$ is equivariant. If $P$ is adapted to $G_r$, then so is $Q$.
\end{itemize}
\end{lem}  
Part (a) is the converse of \cite[Lemma 3.1]{L2}.

\begin{proof}
(a)\ Let $A$ be the collection of $\bC$ valued 1-forms $\alpha$ on $N$ of class $C^1$ that vanish on $Q$. Involutivity means that $d\alpha | Q$ also vanish for all $\alpha\in A$. Now $P=\bigcap_{\alpha\in A}\Ker\,\psi^*\alpha$ is known to be involutive, hence for $x\in M$
\[
0=(d \psi^*\alpha)| P_x=(\psi^* d\alpha) | P_x, 
\]
i.e., $d\alpha | Q_{\psi(x)}=0$, as needed.

(b) Write $\Omega_xg=xg$ for $x\in M$ or $N$ and $g\in G_r$. That $P$ is adapted means $\Omega_{x*}S\subset P$. Since $\psi\circ\Omega_x=\Omega_{\psi(x)}$, 
\[\Omega_{\psi(x)*}S=\psi_*\Omega_{x*}S\subset \psi_*P=Q,\]
which in turn means $Q$ is adapted.
\end{proof}

Given $p\in [1,\infty]$, consider next a $C^{p+1}$ manifold $X$ with a Riemannian metric $h$ of class $C^p$. On the tangent bundle $\pi: TX\to X$ the tautological 1-form and the symplectic 2-form are defined by
\[\alpha(\xi)=h(v,\pi_*\xi), \quad \xi\in T_vTX,\qquad\text{and}\quad \omega=d\alpha.\]
Let $(Y,k)$ be a similar Riemannian manifold with bundle projection $\varrho: TY\to Y$ and corresponding forms $\beta$ and $\sigma=d\beta$ on $TY$. If $\vp: X\to Y$ is a submersion of class $C^{p+1}$, $\Ker\, \vp_*$ is a $C^p$ subbundle of $TX$, whose orthogonal complement is the bundle $H\to X$ of horizontal vectors.

\begin{lem}   
If $\vp:X\to Y$ is a Riemannian submersion and $\Phi=\vp_*|H:H\to TY$, then $\alpha|H=\Phi^*\beta$ and $\omega|H=\Phi^*\sigma$.
\end{lem} 

\begin{proof}
Any $u\in TX$ can be written $u=u_0+u_H$ with $u_0\in\Ker\,\pi_*$ and $u_H\in H$. With $v\in H$ we have
\[
h(v,u)=h(v,u_H)=k(\vp_*v, \vp_*u_H)=k(\vp_*v, \vp_*u).
\]
If $v\in H$ and $\xi\in T_vTX$, in light of $\varrho\circ\Phi=\vp\circ\pi$ therefore
\[\beta(\Phi_*\xi)=k(\Phi v,\varrho_*\Phi_*\xi)=k(\vp_*v, \vp_*\pi_*\xi)=h(v,\pi_*\xi)=\alpha(\xi).\]
Thus $\alpha|H=\Phi^*\beta$ and $d\alpha| H=\Phi^*d\beta$.
\end{proof}
\begin{proof}[Proof of Theorem 3.1] 
For a differentiable map $f:U\to V$ we have used, and will continue to use $f_*$ to denote its differential acting between 
$\bC TU$ and $\bC TV$. It will be convenient to denote the restriction $f_*| TU$ by $f_*^\bR$. 

The open embeddings (3.1) pull back the symplectic forms $\omega,\sigma$ to $M_r$ and $N_r$. We will keep calling the pull backs $\omega$ and $\sigma$. Since $H_r\subset M_r$ corresponds to $H\subset TX$, 
\begin{equation}  
\psi^*\sigma=\omega| H_r.
\end{equation}   
by Lemma 3.3. First we show that $\psi$ is a submersion and
\begin{equation}  
\Ker\, \psi_*^\bR=\{\xi\in T_xH_r: x\in H_r\,\text{ and }\, \omega(\xi,\eta)=0\,\text{ for all }\, \eta\in T_xH_r\}.
\end{equation}   

Indeed, given $x_0\in H_r$, construct through $x_0(0)\in X$ a submanifold $X^\perp$ of a neighborhood of $x_0(0)$ that 
$\vp$ maps diffeomorphically on a neighborhood of $\vp\big(x_0(0)\big)\in Y$. Any $y\in N_r$ in a suitable neighborhood of 
$\psi(x_0)$ has a unique lift to a horizontal geodesic $x$ in $X$ with $x(0)\in X^\perp$, and the $C^1$ map $y\mapsto x$ is a 
right inverse to $\psi$ in this neighborhood; hence $\psi$ is a submersion. Further, if $\xi,\eta\in T_xH_r$, by (3.2)
$\omega(\xi,\eta)=0$ is equivalent to $\sigma(\psi_*\xi,\psi_*\eta)=0$. Since $\psi_*T_xH_r=T_{\psi(x)}N_r$, (3.3) follows.

Next we recall that according to \cite[Corollary 5.5 and Theorem 5.6]{LSz1} $-\omega$ is a K\"ahler form on $M$ when $M$ is endowed with its adapted complex structure. Let $J: TM\to TM$ be the almost complex tensor of the adapted complex structure and let $\perp$ denote orthogonal complement in the Hermitian metric $\omega(J\xi,\eta)$ on $TM$. By (3.3)
\[J \Ker\, \psi_*^\bR\subset T(M\cap H_r)^\perp.\]
We compute
\begin{gather*}
\dim M_r=2\dim X, \quad\dim N_r=2\dim Y,\quad \dim H_r=\dim X+\dim Y,\\
\text{rk}\, T(M\cap H_r)^\perp=\dim X-\dim Y,\quad \text{rk}\,\Ker\, \psi^\bR_*=\dim H_r-\dim N_r=\dim X-\dim Y,
\end{gather*}
whence in fact $J \Ker\,\psi_*^\bR=T(M\cap H_r)^\perp$. This implies
\begin{equation}  
\begin{aligned}
E=(\Ker\,\psi_*^\bR\oplus J\Ker\,\psi_*^\bR)^\perp&=T(M\cap H_r)\cap J T (M\cap H_r)\quad \text{and} \\
\{\xi-iJ\xi:\xi\in E\}&= P\cap \bC T(M\cap H_r).
\end{aligned}
\end{equation}  
This latter bundle, that we denote $R$, is involutive since $P$ is, and endows $M\cap H_r$  with an involutive structure of 
class $C^{p-1}$ (in fact, a CR structure), that is clearly adapted to $G_r$. Furthermore, $\psi_*$ maps 
$E_x=T_xH_r\cap(\Ker\,\psi_*^\bR)^\perp$ 
isomorphically on $T_{\psi(x)}N$ for $x\in M\cap H_r$, and so
\begin{equation}   
\psi_*R_x\oplus \psi_*\overline{R}_x=\bC T_{\psi(x)} N.
\end{equation}  
We claim that $\psi_*R_x$  depends only on $\psi(x)$ or equivalently, whether $\xi-i\eta\in R_x$ for $\xi,\eta\in T_xH_r$ depends only on $\psi_*\xi$, $\psi_*\eta$. 

We will show this by a variant of the proof of \cite[Theorem 4.2]{LSz1}. The actions of $G_r$ on $M_r, N_r$ induce 
actions on $\bC TM_r, \bC TN_r$, that we denote $(\xi,g)\mapsto \xi g$. Various objects introduced in this proof are 
compatible with the $G_r$ actions: $\psi$ and $\psi_*$ are equivariant and $E\subset T(M\cap H_r)$ is invariant. As to this 
latter, note that $g(t)=a+t$ acts by the geodesic flow, hence preserves $\omega$, while $g(t)=bt$ acts by scaling on 
$M\subset TX$, hence changes $\omega$ to a multiple.
Therefore
\[
E=\{\eta\in T_x H_r: x\in M\cap H_r\,\text{ and }\, \omega(\xi,\eta)=0\,\text{ for all } \xi\in(\Ker\,\psi_*)_x\}
\]
is indeed invariant.

If $\xi\in TM$, write $\xi'=(\xi-iJ\xi)/2\in P$. Fix $x\in M\cap H_r$ and choose 
$\xi_1,\dots,\xi_m\in T_xM$ so that $\xi'_1,\dots,\xi'_m$, resp. $\xi'_1,\ldots,\xi'_n$ form a basis of $P_x$, resp. $R_x$ 
($m=\dim X=\text{rk}_\bC\,P$, $n=\dim Y=\text{rk}_\bC\, R$). In particular, $\xi_\alpha\in E_x$ for $\alpha\le n$ 
are independent, and form
a basis. Next choose $\eta_1,\dots,\eta_n\in E_x$.
By \cite[Proposition 5.1]{LSz1} the continuous maps
\[
G_r\ni g\mapsto (\xi_\alpha g)',\, (\eta_\alpha g)' \in P=T^{10}M
\]
are holomorphic over int$\, G_r$. Hence for $g\in \text{int}\,G_r$, except perhaps a discrete subset 
$\Delta\subset \text{int}\,G_r$, the vectors $(\xi_\alpha g)'$ are linearly independent. Therefore there are meromorphic 
functions $F_{\alpha\beta}$ on $\text{int}\, G_r$ such that
\begin{equation}   
(\eta_\alpha g)'=\sum^m_{\beta=1} F_{\alpha\beta}(g)(\xi_\beta g)', \qquad \alpha=1,\dots, n, \quad g\in \text{int}\, G_r.
\end{equation}  
Moreover, as the $(\xi_\beta g)'$ are independent at $g=e$ as well, $F_{\alpha\beta}$ have a limit at $e$, that we denote 
$F_{\alpha\beta}(e)$. Let $\Gamma\subset G_r$ consist of constant maps $t\mapsto a$, where $|a|<r$. Since $N_rg=Y$ 
when $g\in\Gamma$, 
\[\psi_*(\xi_\alpha g),\, \psi_*(\eta_\alpha g)\in T_{xg} Y, \qquad g\in\Gamma.\]
In addition, $\xi_{\alpha}g, \eta_\alpha g\in E_{xg}$, and $\xi_{\alpha}g$ for $\alpha=1,\dots, n$ are linearly independent 
when $g\in\Gamma\setminus\Delta$ (because the $(\xi_\alpha g)'$ are).  Therefore there are unique functions 
$f_{\alpha\beta}:\Gamma\setminus\Delta\to\bR$, $\alpha,\beta\le n$, such that
\begin{equation}   
\eta_\alpha g=\sum\limits^n_{\beta=1} f_{\alpha\beta}(g)\xi_{\beta} g, \quad\text{ or equivalently, }
\quad \psi_*\eta_\alpha g=\sum\limits^n_{\beta=1}f_{\alpha\beta}(g)\psi_*\xi_{\beta} g,
\end{equation}  
when $g\in\Gamma\setminus\Delta$. It follows that $(\eta_\alpha g)'=\sum_{\beta=1}^n f_{\alpha\beta}(g)(\xi_\beta g)'$, 
and by (3.6) that $F_{\alpha\beta}=f_{\alpha\beta}$ on $\Gamma\setminus\Delta$ for $\alpha,\beta\le n$. Thus the meromorphic 
functions $F_{\alpha\beta}$ for $\alpha,\beta\le n$ are uniquely determined by $f_{\alpha\beta}$ which, in view of (3.7), 
are determined by $\psi_*\xi_\alpha, \psi_*\eta_\alpha$. But $\xi_\alpha+i\eta_\alpha\in R_x$ if and only if 
$\eta'_\alpha=i\xi'_\alpha$ ($\alpha\le n)$. Since $\eta'_1,\dots,\eta'_n\in R_x$ and $\xi'_1,\ldots,\xi'_n$ is a basis,
(3.6) implies
$$
\eta'_\alpha=\sum_{\beta=1}^n F_{\alpha\beta}(e)\xi'_\beta,\qquad\alpha=1,\ldots,n.
$$
Hence $\xi_\alpha+i\eta_\alpha\in R_x$ for $\alpha\le n$ if and only if  
$F_{\alpha\beta}(e)=i\delta_{\alpha\beta}$ ($\alpha,\beta\le n$), and depends only on $\psi_*\xi_\alpha$, 
$\psi_*\eta_\alpha$, and their orbit under $G_r$.

Therefore $\psi_*R_x=Q_{\psi(x)}$ indeed depends only on $\psi(x)$ and defines a $C^{p-1}$ subbundle  $Q\subset\bC TN$, which, by (3.5) and by Lemma 3.2 is an adapted complex structure on $N$. 
\end{proof}

\section{Riemannian submersions and analyticity}

Finally we prove the following version of Theorem 1.1. 
\begin{thm}  
Let $p\in (3, \infty]$ be nonintegral, $Y$ a $C^{p+1}$ manifold endowed with a $C^p$ Riemannian metric $k$, and $(X,h)$ an analytic Riemannian manifold. If $\vp:X\to Y$ is a surjective Riemannian submersion of class $C^{p+1}$, then $Y$ admits an analytic structure in which $h$ is analytic. Moreover, this analytic structure and the original $C^{p+1}$ manifold structure of $Y$ share the same underlying $C^p$ manifold structure.
\end{thm}   

\begin{proof}
Fix $r\in(0,\infty)$. Let $M_r, N_r$ be the space of geodesics $[-r,r]\to X$, resp. $Y$, into which $X$ and $Y$ are embeded 
as the space of constant geodesics. Possibly after shrinking $X$ and $Y$, some neighborhood $M$ of $X\subset M_r$ 
will admit an analytic adapted complex structure $P$ (Theorem 2.4). By Theorem 3.1 $P$ induces an adapted complex 
structure $Q$ of class $C^{p-1}$ on a neighborhood $N\subset N_r$ of $Y$. The conclusion of the theorem now follows 
from Theorem 2.5.
\end{proof}

It is a natural question what happens if we weaken the assumption of Theorem 4.1 on $(X,h)$ to $C^\infty$ smoothness. Does it imply that $(Y,k)$ must be $C^\infty$? We do not know the answer.

\end{document}